\documentclass[12pt,reqno]{amsart}
\advance \topmargin by -\headheight
\advance \topmargin by -\headsep
\evensidemargin \oddsidemargin
\usepackage{xcolor}

\usepackage{amsmath}
\usepackage{amsfonts}
\usepackage{stmaryrd}
\usepackage{amssymb}
\usepackage{amsthm}
\usepackage{csquotes}

\usepackage{mathrsfs}
\usepackage{dsfont}
\usepackage{bm}
\usepackage{cite}
\usepackage{soul}

\numberwithin{equation}{section}
\renewcommand\vec{\bm}
\newcommand{\n}[1]{\|{#1}\|}

\newtheorem{theorem}{Theorem}[section]

\newtheorem{lemma}[theorem]{Lemma}
\newtheorem{Proposition}[theorem]{Proposition}
\newtheorem{Conjecture}[theorem]{Conjecture}

\usepackage{mathtools}
\DeclarePairedDelimiter{\ceil}{\lceil}{\rceil}

\title[An Elekes--R\'{o}nyai theorem for sets with few products]{An Elekes--R\'{o}nyai theorem for sets with few products}

\author[Akshat Mudgal]{Akshat Mudgal}
\address{Mathematical Institute, University of Oxford, Oxford OX2 6GG, UK}
\email{mudgal@maths.ox.ac.uk}


\renewcommand\vec{\bm}

\begin{document}

\begin{abstract}
Given $d,n \in \mathbb{N}$, we write a polynomial $F \in \mathbb{C}[x_1,\dots,x_n]$ to be degenerate if there exist $P\in \mathbb{C}[y_1, \dots, y_{n-1}]$ and $m_j = x_1^{v_{j,1}}\dots x_n^{v_{j,n}}$ with $v_{j,1}, \dots, v_{j,n} \in \mathbb{Q}$, for every $1 \leq j \leq n-1$, such that $F = P(m_1, \dots, m_{n-1})$. Our main result shows that whenever $F$ is non-degenerate, then for every finite set $A\subseteq \mathbb{C}$ such that $|A\cdot A| \leq K|A|$, one has
\[ |F(A, \dots, A)| \gg_{d,n} |A|^n 2^{-O_{d,n}((\log 2K)^{3 + o(1)})}. \]
This is sharp up to a factor of $O_{d,n,K}(1)$ since we have the upper bound $|F(A,\dots,A)| \leq |A|^n$ and the fact that for every degenerate $F$ and finite set $A \subseteq \mathbb{C}$ with $|A\cdot A| \leq K|A|$, one has
\[ |F(A,\dots,A)| \ll K^{O_F(1)}|A|^{n-1}.\]
Our techniques rely on a variety of combinatorial and linear algebraic arguments combined with Freiman type inverse theorems and Schmidt's subspace theorem. 
\end{abstract}

\maketitle

\section{Introduction}

This paper concerns the behaviour of polynomials over grids in complex spaces. In particular, given a positive integer $n$, a non-constant polynomial $F \in \mathbb{C}[x_1, \dots, x_n]$ and a finite set $A \subseteq \mathbb{C}$, we are interested in studying properties of the set
\[ F(A, \dots, A) = \{ F(a_1, \dots, a_n) : a_1, \dots, a_n \in A \} . \]
Writing $d$ to be the total degree of $F$, one has the estimate
\begin{equation} \label{es4}
 |A| \ll_{d} |F(A, \dots, A)| \leq  |A|^n,
 \end{equation}
where the upper bound is trivial and the lower bound follows from a straightforward application of the well-known Schwarz-Zippel lemma. A natural question then is to analyse the type of conditions that $F$ must satisfy so as to ensure that 
\[ |F(A, \dots, A)| \gg_{d,n} |A|^{1 + c} \]
for every finite subset $A \subseteq \mathbb{C}$, where $c = c(d,n) >0$ is some constant. Such a line of inquiry was initiated by Elekes--R\'{o}nyai \cite{ER2000} for polynomials $F \in \mathbb{R}[x_1, x_2]$ and sets $A \subseteq \mathbb{R}$, and their work has since been greatly generalised and improved by many authors \cite{BB2021, BT2012, JRT2022, RSS2016, RSZ2016, RS2020}, in part due to its connections to topics in combinatorial geometry  \cite{RSS2016} and model theory  \cite{BB2021}.
\par

In our setting, the above progress has culminated to produce the following result which can be deduced from combining the work of \cite{SZ2022} and \cite{RS2020}. Given $n \geq 2$ and $F \in \mathbb{C}[x_1, \dots, x_n]$ with $\deg F = d$ such that $F$ depends non-trivially in each of $x_1, \dots, x_n$, we either have
\[ |F(A, \dots, A)| \gg_{d,n} |A|^{3/2} \]
for every finite set $A \subseteq \mathbb{R}$, or $F$ is of the form 
\begin{equation} \label{grp}
F(x_1, \dots, x_n) = h(p_1(x_1) + \dots + p_n(x_n)) \ \ \text{or} \ \ F(x_1, \dots, x_n) = h(p_1(x_1)  \dots  p_n(x_n)) 
\end{equation}
for univariate polynomials $h(x), p_1(x), \dots, p_n(x)$. While the above lower bound is far from the upper bound presented in $\eqref{es4}$, there are natural limitations to the kind of quantitative exponents that one can obtain for such results. For example, consider the polynomial $F(x_1, \dots, x_n) = x_1 x_2 + x_3 + \dots + x_n$ and the set $A = \{1, \dots, N\}$. Here, one can see that $F$ does not satisfy either of the forms presented in $\eqref{grp}$ and that 
$|F(A, \dots, A)| \ll_{n} |A|^2.$
\par

Hence, one may ask whether there exists some large family of sets $A$, such that upon restricting our analysis to such sets, one may obtain lower bounds for $|F(A, \dots, A)|$ that are closer to the upper bound mentioned in $\eqref{es4}$. A natural candidate of this type arises in connection to the sum-product conjecture. Thus, given $n \in \mathbb{N}$ and a finite set $A \subseteq \mathbb{C}$, we define the $n$-fold sumset $nA$ and the $n$-fold product set $A^{(n)}$ as 
\[ nA =   \{ a_1 + \dots + a_n : a_1, \dots, a_n \in A\} \ \ \text{and} \ \ A^{(n)}  = \{a_1 \dots a_n : a_1, \dots, a_n \in A\}.\]
The sum-product conjecture, as proposed by Erd\H{o}s and Szemer\'{e}di \cite{ES1983}, states that for any $n \in \mathbb{N}$, any $\varepsilon >0$ and any finite set $A \subseteq \mathbb{Z}$, one should have
\[   |nA| + |A^{(n)}| \gg_{n, \varepsilon} |A|^{n - \varepsilon} .     \]
While this problem remains widely open, it has led to a significant body of work which studies, more generally, the incongruence between additive and multiplicative structure. A striking result in this direction was given by Chang \cite{Ch2003}, who showed that whenever $A\subseteq \mathbb{Z}$ has the product set $A\cdot A = \{ a \cdot b : a, b \in A\}$ satisfying $|A\cdot A| \leq K|A|$ for some $K \geq 1$, then
\[|nA| \gg |A|^n / (2n^2)^{nK}. \]
Firstly, note that when $K$ is small, this almost matches the upper bound $|nA| \leq |A|^n$. Moreover, a quantitatively improved version of this was a crucial ingredient in the breakthrough work of Bourgain--Chang \cite{BC2004} on the sum-product conjecture.
\par

Since the above two results, various works have analysed such expander problems where one focuses on sets with small product sets, see \cite{Ch2006, CH2010a, CH2010b, HRNR2022, HRNZ2019, HRNZ2020, Mu2021c, Mu2023, PZ2020, MRSS2019, Po2020}. For instance, Pohoata \cite{Po2020} showed that whenever $F \in \mathbb{R}[x_1,x_2]$ has $\deg F = d$ and is not of the form $F = P(m(x_1, x_2))$, with $P \in \mathbb{R}[y]$ and $m \in \mathbb{R}[x_1,x_2]$ being a monomial, then $|F(A,A)| \gg_{d,K} |A|^2$ for every finite $A \subseteq \mathbb{R}$ with $|A\cdot A| \leq K|A|$. Similarly, Hanson--Roche-Newton--Zhelezov \cite{HRNZ2019} proved that if one sets $F(x_1, \dots, x_n) = (x_1 + u)\dots(x_n + u)$ for any $u \in \mathbb{Q}\setminus \{0\}$, then for every finite $A \subseteq \mathbb{Q}$ with $|A\cdot A| \leq K |A|$, one has $|F(A, \dots, A)| \gg_{n} |A|^n 2^{-O_n(K)}$.
\par

In this paper, we generalise and strengthen the aforementioned results of \cite{Ch2003, Po2020, HRNZ2019} for a much more extensive class of polynomials and for sets $A \subseteq \mathbb{C}$ with few products. Thus, given $\vec{v} = (v_1, \dots, v_n) \in \mathbb{R}^n$ and $\vec{x} = (x_1, \dots, x_n)$, we denote $\vec{x}^{\vec{v}} = x_1^{v_1} \dots x_n^{v_n}$. We write $F \in \mathbb{C}[x_1, \dots, x_n]$ to be \emph{degenerate} if 
\begin{equation} \label{degf}
 F = P(\vec{x}^{\vec{v}_1}, \dots, \vec{x}^{\vec{v}_{n-1}}), \ \text{for some} \ P \in \mathbb{C}[y_1, \dots, y_{n-1}] \   \text{and}   \ \vec{v}_1, \dots, \vec{v}_{n-1} \in \mathbb{Q}^n. 
 \end{equation}
Denoting $F$ to be \emph{non-degenerate} if $F$ is not degenerate, we now state our main result.

\begin{theorem} \label{main}
Given $d,n \in \mathbb{N}$, there exists a constant $C = C(d,n)>0$ such that for every non-degenerate $F \in \mathbb{C}[x_1,\dots,x_n]$ with $\deg F = d$  and for every finite set $A \subseteq \mathbb{C}$ with $|A \cdot A| \leq K|A|$ for some $K \geq 1$, one has
\begin{equation} \label{qs}
|F(A, \dots, A)| \gg_{d,n} \frac{|A|^n}{2^{C ( \log 2K)^{3 +o(1)}}}. 
\end{equation}
\end{theorem}

Theorem \ref{main} recovers and quantitatively strengthens the main results of \cite{Ch2003, Po2020, HRNZ2019}. In order to see this, we note that if some $F \in \mathbb{R}[x_1,x_2]$ satisfies $F \neq P(m(x_1, x_2))$ for every $P \in \mathbb{R}[y]$ and every monomial $m \in \mathbb{R}[x_1,x_2]$, then $F$ is non-degenerate; such a conclusion may be derived from the proof of Proposition \ref{rez}. Next, it is straightforward to observe that the polynomial $F(\vec{x}) = x_1 + \dots + x_n$ is always non-degenerate. Finally, for any $u \in \mathbb{Q} \setminus \{0\}$, the polynomial $F(\vec{x}) = (x_1 + u)\dots(x_n + u)$ can be seen to be non-degenerate, see,  for instance, Proposition \ref{rez}.  We further note that the quantitative dependence of our main result on $K$ is better than that of these previous results owing to our application of Sanders' strengthening of Freiman's theorem \cite{Sa2013} in our proof of Theorem \ref{main}. Indeed our lower bound in \eqref{qs} decays quasipolynomially in $K$, while the previous results exhibited exponential decay in $K$. The $(\log 2K)^{o(1)}$ factor in the exponent in \eqref{qs} means that for any $\varepsilon>0,$ we have $(\log 2K)^{o(1)} \ll_{\varepsilon} (\log 2K)^{\varepsilon}$, and in fact, one can replace this with $(\log \log 16K)^{O(1)}$ as mentioned at the beginning of \cite[\S12]{Sa2013}.

As before, when $K$ is small, the lower bound in \eqref{qs} matches the upper bound presented in \eqref{es4} up to a multiplicative constant. Moreover, the hypothesis of Theorem \ref{main} seems to be optimal since the requirement that $F$ be non-degenerate is not only sufficient but also necessary. This is recorded as follows.

\begin{Proposition} \label{opt}
Let $A \subseteq \mathbb{C}\setminus \{0\}$ be a finite, non-empty set satisfying $|A \cdot A| \leq K|A|$ for some $K \geq 1$ and let $F \in \mathbb{C}[x_1, \dots, x_n]$ be a degenerate polynomial with degree $d$. Then
\[ |F(A, \dots, A)| \ll K^{O_{F}(1)} |A|^{n-1}.  \]
\end{Proposition}

An instance of this can be illustrated by setting $F(\vec{x}) = x_1 + \dots + x_{n-2} + x_{n-1} x_n$ and $A = \{2,4,\dots, 2^N\}$. Here, note that $F$ is degenerate and $|A \cdot A| \leq 2|A|$, and that 
\[ |F(A,\dots,A)| \leq |A|^{n-2} |A\cdot A| \leq 2 |A|^{n-1}.  \]
\par

It is worth mentioning that while Theorem \ref{main} works in the more general setting of complex numbers, an analogous statement for $F \in \mathbb{R}[x_1,\dots,x_n]$, with the degeneracy statement only allowing for $P \in \mathbb{R}[y_1,\dots,y_{n-1}]$, can be retrieved from our methods mutatis mutandis.  Moreover,  while the property of being non-degenerate depends on the choice of $n$,  this will always be clear from context.

The proof of Theorem \ref{main} relies on a variety of combinatorial and linear algebraic arguments amalgamated with several results from additive combinatorics and number theory,  including two crucial ingredients in the form of 
Freiman's theorem \cite{Sa2013} and Schmidt's subspace theorem \cite{ESS2002}. The latter two types of results seem to have been first combined  to analyse the sum-product problem by Chang \cite{Ch2006}. The novelty of our methods lies in the application of the former techniques which allow for Theorem \ref{main} to hold in such a generality while also significantly simplifying the applications of the subspace theorem.  This also enables us to bypass applications of algebraic results from \cite{St1989},  the latter having been crucial for the work of \cite{Po2020, RSS2016}.
\par

 In order to elaborate one such idea, we remark that unlike most of the preceding results in this direction, the ``energy version" of Theorem \ref{main} does not hold true. Thus, we define 
\[ E_F(A) = |\{ (a_1,\dots, a_{2n}) \in A^{2n} : F(a_1, \dots, a_n) = F(a_{n+1},\dots, a_{2n}) \}|,\]
and note that a straightforward application of the Cauchy-Schwarz inequality gives us
\[ |F(A,\dots,A)| \geq |A|^{2n} E_{F}(A)^{-1}. \]
Consequently, in order to show that $|F(A,\dots,A)| \gg_{d,n,K} |A|^n$ for some finite set $A \subseteq \mathbb{C}$ with $|A \cdot A| \leq K|A|$, it suffices to prove a bound of the shape
\begin{equation} \label{enbd}
E_{F}(A) \ll_{d,n,K} |A|^n .
\end{equation}
Indeed, previous results from \cite{Ch2003, Ch2006, HRNZ2019, HRNZ2020, PZ2020, Mu2023, Po2020} proceeded via this direction and were able to show that \eqref{enbd} holds for sets $A$ with $|A\cdot A| \leq K|A|$, for some subset $A$ of either $\mathbb{Q},\mathbb{R}$ or $\mathbb{C}$, and for a variety of special cases of $F$ and $n$. This can be seen to be in line with the closely related Elekes-Szab\'{o} philosophy \cite{ES2012}, where instead of finding lower bounds for $|F(A, \dots,A)|$, one is interested in finding upper bounds for $|\{ \vec{a} \in A^n : F(\vec{a}) = 0\}|$.
\par

On the other hand, \eqref{enbd} fails to hold in the more general setting in which Theorem \ref{main} may be applied. For instance, given $n \geq 5$, let $f_0, f_1, \dots, f_{n-2} \in \mathbb{Z}[x_1,x_2] \setminus (\mathbb{Z}[x_1] \cup \mathbb{Z}[x_2])$ be polynomials which are not monomials, such that
\[ f_0(a,b) = f_1(a,b) = \dots = f_{n-2}(a,b) = 0, \]
for some $a,b \in \mathbb{Z} \setminus \{0\}$. Setting
\begin{equation}  \label{algo}
F(x_1,\dots, x_n ) = f_0(x_1, x_2) + x_3 f_1(x_1, x_2) + \dots + x_n f_{n-2}(x_1, x_2), 
\end{equation}
one may verify that $F$ is not degenerate, for example, through the medium of Proposition \ref{rez}. Moreover, given any finite set $A \subseteq \mathbb{Z}$ with $a,b \in A$, we see that 
\[ F(a,b, a_3,\dots, a_n) = F(a,b,a_3',\dots,a_n') = 0 \]
for every $a_3, \dots, a_n, a_3',\dots, a_n' \in A$, whereupon, we have that
\[ E_{F}(A) \geq |A|^{2n-4}. \]
This contradicts \eqref{enbd} whenever $n \geq 5$ and $|A|$ is sufficiently large in terms of $d,n,K$. We remark that while polynomials of the form \eqref{grp} and \eqref{degf} can be considered as additive and multiplicative obstructions, our example from \eqref{algo} can be construed as some kind of an algebraic obstruction which allows for the existence of hidden subvarieties. 
\par

In our setting, we are able to circumvent these algebraic obstructions by replacing $E_F(A)$ with a more flexible object which considers solutions to $F(a_1, \dots, a_n) = m$, for any fixed $m \in \mathbb{C}$, with $(a_1, \dots, a_n)$ lying in a dense hypergraph. This idea combines with our methods to provide the following version of inequality \eqref{enbd}.

\begin{theorem} \label{env}
Let $d,n$ be natural numbers, let $A \subseteq \mathbb{C}^{\times}$ be a finite set with $|A\cdot A|\leq K |A|$ for some $K \geq 1$ and let $F \in \mathbb{C}[x_1, \dots, x_{n}]$ be non-degenerate with $\deg F = d$. Then there exists some set $\mathcal{G} \subseteq A^{n}$ with $|A^{n} \setminus \mathcal{G}| \ll_{d,n} |A|^{n-1}$ such that
\begin{equation} \label{mo}
\sup_{m \in \mathbb{C}} | \{ \vec{a} \in \mathcal{G} : F(\vec{a}) = m\}| \ll_{d,n}  2^{O_{d,n}(K)} .
\end{equation} 
In particular, we have that
\[  |\{ (\vec{a},\vec{a}') \in \mathcal{G}\times\mathcal{G} : F(\vec{a}) = F(\vec{a}') \}|  \ll_{d,n} 2^{O_{d,n}(K)} |A|^{n}.  \]  
\end{theorem}

We note that the dependence on $K$ in the above estimates matches the previous results of \cite{Ch2003, Po2020, HRNZ2019} and is worse than that of Theorem \ref{main}. This is primarily because in the latter, we can pass to a suitably large subset $A'$ of $A$ without losing substantially large factors of $K$, a manoeuvre which does not seem to be possible in the setting of Theorem \ref{env}. Consequently, we forego the application of Sanders' deep results on Freiman's theorem in this setting by exploiting the circle of ideas surrounding Freiman's lemma \cite[Lemma 5.13]{TV2006}, the latter allowing for better quantitative estimates in this regime. In particular, we are required to obtain an upper bound on the rank of multiplicative subgroups generated by finite sets $A \subseteq \mathbb{C}^{\times}$ satisfying $|A\cdot A| \leq K|A|$ and this is done via the means of Lemma \ref{sumset}.

We remark that finding quantitatively optimal results towards Freiman type inverse theorems forms one of the most central problems in additive combinatorics, the key conjecture here being the polynomial Freiman Ruzsa Conjecture, see \cite{Sa2013}. Roughly speaking, this implies that given any finite set $A$ of integers with few pairwise sums, one can find a large subset of $A$ which is contained efficiently in the affine image of a low dimensional convex body; we provide more details concerning this in \S6. Conditional on such conjectures, the methods of our paper can prove Theorem \ref{main} with a polynomial dependence on $K$, see Proposition \ref{chs}. While one motivation towards such improvements arises from the quantitative study of Elekes-R\'{o}nyai problem, a more intimate connection may be found towards the aforementioned sum-product conjecture. In particular, the previously mentioned breakthrough work of Bourgain--Chang \cite{BC2004} implies that for any $k \in \mathbb{N}$, there exists some $n \in \mathbb{N}$ such that for any finite $A \subseteq \mathbb{Q}$, one has
\begin{equation} \label{spc3}
 |nA| + |A^{(n)}|  \gg_{k} |A|^k.
 \end{equation}
While this has been strengthened in \cite{PZ2020} and extended to a more general phenomenon concerning energies of polynomials and products in \cite{Mu2023}, a major open problem in the area is to prove \eqref{spc3} for finite sets $A$ of real numbers. It was shown by Chang \cite{Ch2009} that such problems have close connections to the weak polynomial Freiman Ruzsa Conjecture, see Conjecture \ref{wkpf}, and indeed, using our methods, we are able to  conditionally prove a much more general version of \eqref{spc3}.

\begin{theorem} \label{condn}
    Suppose that Conjecture \ref{wkpf} holds true. Then for every $d,k \in \mathbb{N}$ and $\varepsilon >0$, there exists $n = n(d,k, \varepsilon) \in \mathbb{N}$ such that for every non-degenerate $F \in \mathbb{C}[x_1,\dots,x_{k}]$ with $\deg F = d$  and for every finite set $A \subseteq \mathbb{R}$, one has
\[ |A^{(n)}| + |F(A, \dots, A)| \gg_{d,k} |A|^{k- \varepsilon}.  \]
\end{theorem}

We remark that a version of this result for finite sets $A\subseteq \mathbb{Q}$ and for non-degenerate $F \in \mathbb{Q}[x_1,\dots,x_{k}]$ of the form \eqref{grp} can be deduced unconditionally from our work in \cite{Mu2023}.

We now provide a brief outline of our paper. We begin by utilising \S2 to prove some key properties of degenerate polynomials in the form of Propositions \ref{opt} and \ref {rez}. We employ \S3 to record various results from additive combinatorics and number theory that we will require for our proofs of Theorems \ref{main} and \ref{env}. In \S4, we will present the proof of Lemma \ref{spt}, which delivers an analogue of Theorem \ref{env} for sets $A \subseteq \mathbb{C}^{\times}$ that generate a multiplicative subgroup of small rank. This naturally combines with Freiman type inverse theorems to dispense our main results, which is what we pursue in \S5. Finally, in \S6, we prove Theorem \ref{condn}.

\textbf{Notation.} In this paper, we use Vinogradov notation, that is, we write $X \gg_{z} Y$, or equivalently $Y \ll_{z} X$, to mean $X \geq C_{z} |Y|$ where $C_z$ is some positive constant depending on the parameter $z$. We further write $X = O_{z}(Y)$ to mean $X \ll_{z} Y$. For every natural number $k \geq 2$ and for every non-empty, finite set $Z$, we use $|Z|$ to denote the cardinality of $Z$, we write $Z^k = \{ (z_1, \dots, z_k)   :   z_1, \dots, z_k \in Z\}$ and we use boldface to denote vectors $\vec{z} = (z_1, z_2, \dots, z_k) \in Z^k$. Furthermore, given $\vec{v} \in \mathbb{R}^k$, we denote the monomial $\vec{x}^{\vec{v}} = x_1^{v_1} \dots x_k^{v_k}$, and given positive integer $n$, we define $[n]=\{1,2,\dots,n\}$.  Given a field $\mathbb{F}$,  we write $\mathbb{F}^{\times} = \mathbb{F} \setminus \{0\}$ to be the set of invertible elements of $\mathbb{F}$.  Finally, all of our logarithms are with respect to base $2$.

\textbf{Acknowledgements}. The author is supported by Ben Green's Simons Investigator Grant, ID 376201. The author would like to thank Ben Green for helpful discussions.


\section{Properties of degenerate polynomials}

We begin this section by presenting some notation, and so, we write $E = \mathbb{N} \cup\{0\}$ 
and $\vec{x}^{0} = 1$. Next, given a non-zero polynomial $F \in \mathbb{C}[x_1, \dots, x_n]$, we define $\mathcal{I}_F \subseteq E^n$ to be the unique set which satisfies
\[ F(\vec{x}) = \sum_{\vec{i} \in \mathcal{I}_F} c_{\vec{i}} \vec{x}^{\vec{i}},\]
where $c_{\vec{i}} \in \mathbb{C} \setminus \{0\}$ whenever $i \in \mathcal{I}_F$. We denote $d(F)$ to be the dimension of the linear span of the set $\mathcal{I}_F$ over $\mathbb{Q}$. Note that if $d(F) = 1$ for some $F \in \mathbb{C}[x_1,x_2]$, then $F = P(m(x_1, x_2))$, where $P \in \mathbb{C}[x]$ is a univariate non-constant polynomial and $m(x_1, x_2)$ is a monomial. Similarly if $d(F) = 0$ for any $F \in \mathbb{C}[x_1,\dots,x_n]$, then $F$ is just some constant polynomial. An important ingredient in our proofs of Theorems \ref{main} and \ref{env} would be the idea that for any non-degenerate $F \in \mathbb{C}[x_1, \dots, x_{n}]$, we have $d(F) = n$. This will be encapsulated in the following result.

 \begin{Proposition} \label{rez} 
Let $n \geq 1$ be an integer,  let $F\in \mathbb{C}[x_1, \dots, x_n]$ be a non-zero polynomial. Then $d(F) \leq n-1$ holds if and only if $F$ is degenerate.
 \end{Proposition}

\begin{proof}
Note that if $F$ is degenerate, then $F = P(\vec{x}^{\vec{v}_1}, \dots, \vec{x}^{\vec{v}_{n-1}})$, for some polynomial $P \in \mathbb{C}[y_1, \dots, y_{n-1}]$ and some $\vec{v}_1, \dots, \vec{v}_{n-1} \in \mathbb{Q}^n$. Thus, we have that $\mathcal{I}_F$ lies in the subspace generated by $\vec{v}_1, \dots, \vec{v}_{n-1}$, whence, $d(F) = \dim \mathcal{I}_F \leq n-1$. 
\par

We now prove that whenever $d(F) \leq n-1$, then $F$ is degenerate. Thus, suppose that $d(F) = r$ for some $r \leq n-1$. Our aim is to show that there exist elements $\vec{v}_1, \dots, \vec{v}_{n-1} \in \mathbb{Q}^n$ such that any $\vec{v} \in \mathcal{I}_F$ may be represented as
\begin{equation} \label{cnc2}
     \vec{v} = \alpha_{1, \vec{v}} \vec{v}_1 + \dots +  \alpha_{n-1, \vec{v}} \vec{v}_{n-1}, 
\end{equation}
for some $\alpha_{1, \vec{v}}, \dots, \alpha_{n-1, \vec{v}} \in E$, since this would directly imply that $F = P(\vec{x}^{\vec{v}_1},\dots, \vec{x}^{\vec{v}_{n-1}})$ for some $P \in \mathbb{C}[y_1, \dots, y_{n-1}]$. In this endeavour, we may add some points from $E^n$ to $\mathcal{I}_F$, if necessary, so as to ensure that the subspace $V$ generated by $\mathcal{I}_F$ over $\mathbb{Q}$ has $\dim(V) = n-1$. This implies that, up to some reordering of coordinates, we can find $\beta_1, \dots, \beta_{n-1} \in \mathbb{Q}$ such that for any $\vec{v} = (v_1, \dots, v_n) \in V$, we have
\[ v_n = \sum_{i=1}^{n-1} \beta_i v_i.  \]
Now, for every $1 \leq j \leq n-1$, we fix $\vec{v}_j \in V$ such that $v_i = 0$ for every $i \in [n-1]\setminus \{j\}$ and $v_j = 1$. We now claim that $\eqref{cnc2}$ holds true with this choice of $\vec{v}_1, \dots, \vec{v}_{n-1}$. In order to see this, we define, for every $1 \leq i \leq n$, the map $\pi_i : V \to \mathbb{Q}$ as $\pi_i(\vec{v}) = v_i$ for each $\vec{v} = (v_1, \dots, v_n) \in V$.  With this in hand, we observe that any $\vec{u} = (u_1, \dots, u_n) \in V$ satisfies
\[ \vec{u} = u_1 \vec{v}_1 + \dots + u_{n-1} \vec{v}_{n-1},\]
since 
\[ \pi_i \big( \sum_{i=1}^{n-1} u_i \vec{v}_i \big) = u_i  \ \ \text{for every} \ 1 \leq i \leq n-1  \]
and
\[  \pi_n \big( \sum_{i=1}^{n-1} u_i \vec{v}_i \big)  =  \sum_{i=1}^{n-1} u_i \pi_n ( \vec{v}_i ) = \sum_{i=1}^{n-1} \beta_i u_i = u_n. \]
Moreover, for any $\vec{u} \in \mathcal{I}_F$, we know that $u_i \in E$ for every $1 \leq i \leq n$, which consequently proves $\eqref{cnc2}$.
\end{proof}

We will now present some further definitions, and so, given finite subsets $A,B$ of some abelian group $G$, we denote the sumset $A+B$ and the difference set $A-B$ as
\[ A+B = \{a+b : a \in A, b \in B\} \ \ \text{and}  \ \ A - B = \{ a - b : a \in A, b \in B\} . \]
This further allows us to write,  for every $k \in \mathbb{N}$,  the $k$-fold sumset 
\[ kA = \{a_1 + \dots + a_k : a_1, \dots, a_k \in A\} . \]
 When $A,B$ are finite subsets of some ring $R$, we further define the product set 
\[ A\cdot B = \{a \cdot b : a \in A, b \in B\}, \]
and if all the elements of $B$ are invertible, we write
\[ B^{-1} = \{ b^{-1} : b \in B\} \ \ \text{and} \ \ A/B  = A \cdot B^{-1}  = \{ a \cdot b^{-1} : a \in A, b \in B\} .\]
Moreover,  we define the $k$-fold product set $A^{(k)} = \{ a_1 \dots a_k : a_1, \dots, a_k \in A\}$ for every $k \in \mathbb{N}$.

With this notation in hand, we record a useful result from additive combinatorics known as the Pl\"{u}nnecke-Ruzsa inequality which allows one to bound many-fold sumsets in terms of the two-fold sumset.

\begin{lemma} \label{pr}
Let $A$ be a finite subset of some abelian group $G$ and let $|A+A| \leq K|A|$ for some $K \geq 1$. Then for every $l,m \in \mathbb{N}\cup \{0\}$, we have
\[ |kA - lA| \leq  K^{k+l}|A|.\]
\end{lemma}

We are now ready to present the proof of Proposition \ref{opt}.

\begin{proof}[Proof of Proposition \ref{opt}]
Since $F$ is degenerate, we have that $F = P(\vec{x}^{\vec{v}_1}, \dots, \vec{x}^{\vec{v}_{n-1}})$ for some $P \in \mathbb{C}[y_1, \dots, y_{n-1}]$ and $\vec{v}_1, \dots, \vec{v}_{n-1} \in \mathbb{Q}^n$. For each $1 \leq j \leq n-1$, let 
\[ X_j = \{ a_1^{v_{j,1}} \dots a_n^{v_{j,n}} : a_1, \dots, a_n \in A\}. \]
Thus, we see that
\[ |F(A, \dots, A)| \leq |P(X_1, \dots, X_{n-1})| \leq |X_1| \dots |X_{n-1}|, \]
whence, it suffices to show that $|X_j| \ll K^{O_F(1)} |A|$ for every $1 \leq j \leq n-1$. Fixing some $1 \leq j \leq n-1$, we let $v_{i,j} = p_i/q_i$ for some $p_i \in \mathbb{Z}$ and $q_i \in \mathbb{N}$, for every $1 \leq i \leq n$. Thus, writing $M = q_1 \dots q_n$ and $r_i = q_1 \dots q_{i-1} p_i q_{i+1} \dots q_n$ for every $1 \leq i \leq n$, we see that
\[ a_1^{v_{j,1}} \dots a_n^{v_{j,n}} = \prod_{i=1}^{n} (a_i^{q_1 \dots q_{i-1} p_i q_{i+1} \dots q_n})^{1/(q_1\dots q_n)} = (a_1^{r_1} \dots a_n^{r_n})^{1/M},    \]
for any $a_1, \dots, a_n \in A$. Let $S = \{ s_1, \dots, s_l\}$ and $T = \{t_1, \dots, t_m\}$ be sets such that
\[ S = \{ r_1, \dots, r_n\}\cap [0, \infty)  \ \text{and} \ T =   \{ r_1, \dots, r_n\}\cap (- \infty,0). \]
With this notation in hand, we have that
\begin{align*}
 |X_j| 
 & \ll_{M} | \{ a_1^{r_1} \dots a_n^{r_n} : a_1, \dots, a_n \in A\}| \leq |A^{(s_1 + \dots + s_l)}/ A^{(|t_1| + \dots + |t_m|)}| \\
 & \leq K^{(|s_1| + \dots + |s_l| + |t_1| + \dots + |t_m|)} |A|, 
 \end{align*}
 with the last inequality following from Lemma \ref{pr}. Thus, we have that $|X_j| \ll K^{O_F(1)}|A|$ for every $1 \leq j \leq n-1$, which allows us to conclude the proof of Proposition \ref{opt}.
\end{proof}


\section{Preparatory lemmata}

Our main aim for this section is to record some results from additive combinatorics and number theory that will be important for the proofs of Theorems \ref{main} and \ref{env}. We begin this endeavour by presenting some further notation, and so, given some $\vec{v}\in \mathbb{R}^n$,  we write $\n{\vec{v}}_1 = |v_1| + \dots + |v_n|$.  Similarly,  given a non-zero rational number $p/q$,  with $(p,q) = 1$,  we denote $N(p/q) = \max\{|p|,|q|\}$. We also set $N(0) = 0$.  We can now extend this definition to vectors in $\mathbb{Q}^n$,  and so,  given $\vec{z} = (z_1, \dots,z_n) \in \mathbb{Q}^n$,  we define $N(\vec{z}) = \max_{1 \leq i \leq n} N(z_i)$.  Next,  given $F \in \mathbb{C}[x_1,\dots, x_n]$ and a finite subset $A \subseteq \mathbb{C}$,  we write
\[ V(F,A) = \{ (a_1,\dots, a_n) \in A^n : F(a_1, \dots, a_n) = 0 \} . \]
\par

With this in hand,  we now record the well-known Schwartz–Zippel lemma.

\begin{lemma} \label{scz}
Let $F \in \mathbb{C}[x_1,\dots,x_n]$ be a non-zero polynomial with $\deg(F) = d$ and let $A$ be a finite subset of $\mathbb{C}$.  Then 
\[ |V(F,A)| \leq d |A|^{n-1}. \]
\end{lemma}

As mentioned before, one can utilise this to show that $|F(A,\dots,A)| \geq |A|/d$ for every non-constant $F \in \mathbb{C}[x_1, \dots, x_n]$, since
\begin{align*}
    |A|^n & = \sum_{m \in F(A, \dots, A)} |\{ \vec{a} \in A^n : F(\vec{a}) = m\}|   \\
    &\leq  \sup_{m \in F(A, \dots, A)} |\{ \vec{a} \in A^n : F(\vec{a}) = m\}| \ |F(A,\dots, A)|   \\
    & \leq d|A|^{n-1} |F(A, \dots, A)| .
\end{align*}

We will also require the following quantitative refinement of the subspace theorem of Evertse, Schmidt and Schlikewei \cite{ESS2002} as proved by Amoroso--Viada in \cite{AV2009}. 

\begin{lemma} \label{subs}
Let $l, r$ be natural numbers, let $\Gamma$ be a subgroup of $\mathbb{C}^{\times}$ of finite rank $r$ and let $c_1, \dots, c_{l}$ be non-zero complex numbers. Then the number of solutions of the equation
\[ c_1 z_1 + \dots + c_l z_l = 1\]
with $z_i \in \Gamma$ and no subsum on the left hand side vanishing is at most $(8l)^{4l^4(l + lr + 1)}$.
\end{lemma}

Here, we denote the multiplicative subgroup $\Gamma \subseteq \mathbb{C}^{\times}$ to have rank $r$ if there exists a finitely generated subgroup $\Gamma_0$ of $\Gamma$, again of rank $r$, such that the factor group $\Gamma/\Gamma_0$ is a torsion group. Furthermore, note that if we are in the setting of Lemma \ref{subs} but instead we are counting solutions to the equation
\[ c_1 z_1 + \dots + c_l z_l = m \]
for some fixed $m \in \mathbb{C}\setminus \{0\}$,
with $z_1, \dots, z_l \in \Gamma$ and no subsum on the left hand side vanishing, one may divide both sides above by $m$, and employ the fact that $m^{-1}z_1,\dots,m^{-1}z_l \in m^{-1} \cdot \Gamma$, where we see that $m^{-1} \cdot \Gamma$ itself is contained in a multiplicative subgroup of $\mathbb{C}^{\times}$ of rank $r+1$. Thus one may bound solutions to more general linear equations without vanishing subsums at the cost of increasing the aforementioned upper bound by a factor of $(8l)^{4l^5}$.

The other major result we will use in our proof of Theorem \ref{main} will be Freiman's theorem. Thus, given some abelian group $G$ and some $k \in \mathbb{N}$,  we denote a finite set $P \subseteq G$ to be a \emph{$k$-dimensional centred convex progression} if there exists some symmetric convex body $Q \subseteq \mathbb{R}^k$ and a homomorphism $\varphi : \mathbb{Z}^k \to G$ such that $\varphi(\mathbb{Z}^k \cap Q) = P$. Moreover, we define a set $X \subseteq G$ to be a \emph{$k$-dimensional centred convex coset progression} if $X = P +H$, where $P$ is a $k$-dimensional centred convex progression and $H$ is a subgroup of G. It is known that such sets have a small doubling, that is, 
\[ |X + X| \leq e^{O(k)} |X|, \]
see \cite[Lemma $4.2$]{Sa2013}. Freiman's theorem for general abelian groups $G$ implies that any finite set $A \subseteq G$, with $|A+A|\leq K |A|$, can be covered by $O_K(1)$ translates of such a progression $X$, with $|X| \ll_K |A|$. There has been a rich history of work towards results of this type, see \cite{Ch2002, GR2007, Sa2013} and the references therein, and for our purposes, we will use the following quantitative version of Freiman's theorem from \cite{Sa2013}.

\begin{lemma} \label{freiman}
Let $A \subseteq G$ be a finite set with $|A+A| \leq K|A|$,  for some $K \geq 1$.  Then there exists some $l \in \mathbb{N}$ with $l \ll (\log 2K)^{3 +o(1)}$, a set $T \subseteq G$ with $|T| \leq e^l$, and a $l$-dimensional centred convex coset progression $X$ with $|X| \leq e^l |A|$ such that $A \subseteq T + X$. 
\end{lemma}

We note that weaker versions of above type of inverse results, wherein one is only interested in bounding the  rank of the additive subgroup generated by some finite set $A$ with $|A+A| \leq K|A|$, can be deduced via more elementary methods. These type of ideas will perform an important role in our proof of Theorem \ref{env}. In particular, one such result will follow from the so-called Freiman's lemma \cite[Lemma 5.13]{TV2006}, and so, we record the latter below.

\begin{lemma} \label{frlem}
  Let $X \subseteq \mathbb{R}^d$ be a finite set which is not contained in a translate of some hyperplane. Then 
\[ |X+X| \geq (d+1)|X|  - d(d+1)/2. \]  
\end{lemma}

Now suppose that we have some non-empty, finite set $A \subseteq \mathbb{R}^{\times}$ with $|A| \geq 2$ satisfying $|A\cdot A| \leq K|A|$ for some $K \geq 1$. Let $A' = a^{-1}\cdot A$ for some $a \in A$ and let $\mathcal{V}$ be the multiplicative subgroup generated by $A$. Furthermore, suppose that $\mathcal{V}$ is isomorphic to $\mathbb{Z}^r$ for some $r \in \mathbb{N}$. Note that this condition implies that $|A'| \geq r$. We may now invoke Lemma \ref{frlem} to deduce that
\[ K|A'| \geq |A' \cdot A'| \geq (r+1)|A'|  - r(r+1)/2 \geq (r+1)|A'|/2,  \]
whenceforth, we have that $r \leq 2K$.  Thus, we see that $A$ is contained in $a\cdot \mathcal{V}$,  which,  in turn,  is contained in a multiplicative subgroup of rank at most $r+1 \leq 3K$.

We end this section by presenting Ruzsa's covering lemma as stated in \cite[Lemma 2.14]{TV2006}.

\begin{lemma} \label{rsco}
Let $A,B$ be finite subsets of some abelian group $G$ such that $|A+B| \leq K|A|$, for some $K \geq 1$. Then there exists a non-empty set $X \subseteq B$ with $|X| \leq K$ such that
\[ B \subseteq X + A -A. \]
\end{lemma}


\section{Estimates for sets in multiplicative groups of small rank}

Our main goal of this section is to prove the following lemma, which will later combine with various Freiman type inverse theorems from \S3 to give us Theorems \ref{main} and \ref{env}.

\begin{lemma} \label{spt}
Let $r,n,d$ be natural numbers, let $\Gamma$ be a multiplicative subgroup of $\mathbb{C}^{\times}$ of rank $r$, let $A \subseteq \Gamma$ be a finite set and let $F \in \mathbb{C}[x_1, \dots, x_{n}]$ have $\deg F = d$ with $d(F) = n$ and $F(0,\dots,0) = 0$. Then there exists some set $\mathcal{G} \subseteq A^n$ with $|A^n \setminus \mathcal{G}| \ll_{d,n} |A|^{n-1}$ such that
\[  \sup_{m \in \mathbb{C}} | \{ (a_1, \dots, a_{n}) \in \mathcal{G}  :  F(a_1, \dots, a_{n}) = m \}| \ll_{d,n} 2^{O_{d,n}(r)}. \]

\end{lemma}

  We proceed by first proving an auxiliary lemma,  which in turn, requires us to present some new notation.  Thus,  given a basis $\mathcal{E} =\{ \vec{v}_1,\dots,\vec{v}_n\}$ of $\mathbb{Q}^n$, some vector $\vec{t} \in \mathbb{C}^n$ and some finite set $A \subseteq \mathbb{C}$,  we define
\[ \sigma(\vec{t},\mathcal{E},  A) = | \{ \vec{a} \in A^n : \vec{a}^{\vec{v}_i} = t_i \ \ (1 \leq i \leq n)      \}|,\]
where we write $\vec{a}^{\vec{u}} = a_1^{u_1}\dots a_n^{u_n}$ for every $\vec{a} \in A^n$ and for every $\vec{u} \in \mathbb{Q}^n$.
\begin{lemma} \label{c2}
Let $\kappa>0$ be a real number, let $\vec{t} \in (\mathbb{C}^{\times})^n$ be a vector, let $\mathcal{Z} \subseteq \mathbb{Q}^n$ be a basis of $\mathbb{R}^n$ such that $N(\vec{z}) \leq \kappa$ for every $\vec{z} \in \mathcal{Z}$ and let $A \subseteq \mathbb{C}^{\times}$ be a finite set. Then
\[ \sigma(\vec{t},\mathcal{Z},A) \ll_{\kappa ,n} 1.  \]
\end{lemma}

\begin{proof}
We begin by considering the matrix $M \in \textrm{GL}_n(\mathbb{Q})$ whose row vectors are $\vec{z}_1, \dots, \vec{z}_n$ in that order.  Now,  since $\mathcal{Z}$ is a basis of $\mathbb{Q}^n$,  for each $1 \leq i \leq n$,  there exist $\alpha_{i,1}, \dots, \alpha_{i,n} \in \mathbb{Q}$ satisfying
\begin{equation} \label{app}
\vec{e}_i = \sum_{j=1}^{n} \alpha_{i,j} \vec{z}_j ,
\end{equation}
where $\vec{e}_1,\dots, \vec{e}_n$ form the canonical basis of $\mathbb{R}^n$. In fact,  $\{\alpha_{i,j}\}_{1\leq i,j \leq n}$ are precisely the entries of $M^{-1} \in \textrm{GL}_n(\mathbb{Q})$,  whence,  $N(\alpha_{i,j}) \ll_{\kappa,n} 1$ for every $1\leq i,j\leq n$.  
Noting $\eqref{app}$, we see that if $\vec{a} \in A^n$ satisfies $\vec{a}^{\vec{z}_i} = t_i$ for every $1 \leq i \leq n$,  then 
\[ a_i = \vec{a}^{\vec{e}_i} = \vec{a}^{\sum_{j=1}^{n} \alpha_{i,j} \vec{z}_j } = t_1^{\alpha_{i,1}} \dots t_n^{\alpha_{i,n}}  \ \ \ ( 1 \leq i \leq n). \]
Since $N(\alpha_{i,j}) \ll_{\kappa} 1$,  we may write $a_i = (t_1^{m_{i,1}} \dots t_n^{m_{i,n}})^{1/Q_i}$,  where $m_{i,1}, \dots, m_{i,n}, Q_i$ are some integers satisfying $|m_{i,1}|, \dots, |m_{i,n}| \ll_{\kappa, n} 1$ and $0<Q_i \ll_{\kappa,n} 1$. Thus,  $a_i$ is a root of some polynomial with fixed coefficients and degree $1 \leq Q_i \ll_{\kappa, n}1$,  and so,  for each $1 \leq i \leq n$,  there are at most $O_{\kappa,n}(1)$ possibilities for $a_i$.  This implies that there are $O_{\kappa,n}(1)$ values of $\vec{a} \in A^n$ such that $\vec{a}^{\vec{z}_i} =t_i$ for every $1\leq i \leq n$, thus concluding the proof of Lemma $\ref{c2}$.
\end{proof}

With this lemma in hand,  we now present our proof of Lemma \ref{spt}.

\begin{proof}[Proof of Lemma \ref{spt}]
We note that the case when $n=1$ follows trivially,  whence,  we may assume that $n \geq 2$. Next,  we note that since $d(F) = n$,  there exists a set $\mathcal{E} \subseteq \mathcal{I}_F$ such that $\mathcal{E}$ is a basis of $\mathbb{Q}^n$.  We fix one such set $\mathcal{E}$.  We now consider a non-empty set $\mathcal{I}' \subseteq \mathcal{I}_F$ and analyse the polynomial $F_{\mathcal{I}'} \in \mathbb{C}[x_1, \dots, x_n]$ which is given by
\[ F_{\mathcal{I}'}(\vec{x}) = \sum_{\vec{v} \in \mathcal{I}'} c_{\vec{v}} \vec{x}^{\vec{v}} .\]
 Since $\mathcal{I}'$ is non-empty,  we see that $F_{\mathcal{I}'}$ is a non-zero polynomial, and so, we can utilise Lemma $\ref{scz}$ to infer that
\[ |V(F_{\mathcal{I}'},A)| \leq d |A|^{n-1}. \]
Defining 
\[ V = \cup_{ \emptyset \neq \mathcal{I}' \subseteq \mathcal{I}_F} V(F_{\mathcal{I}'},A) \ \text{and} \ \mathcal{G} = A^n \setminus V,\]
the preceding inequality gives us
\[ |A|^n - |\mathcal{G}| = |V| \leq \sum_{\mathcal{I}' \subseteq \mathcal{I}_F} |V(F_{\mathcal{I}'},A)|  \ll_{d,n}|A|^{n-1}.  \]
We further denote 
\begin{equation} \label{repdef}
  r_{F, \mathcal{G}}(m)  = |\{ \vec{a} \in \mathcal{G} : F(\vec{a}) = m \}|
\end{equation}
for every $m \in \mathbb{C}$.  Since $V(F, A) \cap \mathcal{G} = \emptyset$,  we see that $r_{F,\mathcal{G}}(0) = 0$, and so, our main aim now is to show that 
\begin{equation} \label{sd3}
 \sup_{m \in \mathbb{C}^{\times}} r_{F,\mathcal{G}}(m) \ll_{d,n} 2^{O_{d,n}(r)} .
 \end{equation}
\par

Thus,  given $m \in \mathbb{C}^{\times}$,  suppose that
\[ m = \sum_{\vec{v} \in \mathcal{I}_{F}} c_{\vec{v}} \vec{a}^{\vec{v}}  \]
for some $\vec{a} \in \mathcal{G}$.  We will proceed to show that there are at most $O_{d,n}(2^{O_{d,n}(r)})$ possible such choices of $\vec{a} \in \mathcal{G}$.  In particular,  as $\vec{a} \in \mathcal{G}$,  we can not have any vanishing subsums of the form $\sum_{\vec{v} \in \mathcal{I}} c_{\vec{v}} \vec{a}^{\vec{v}} = 0$ for any $\mathcal{I} \subseteq \mathcal{I}_F$, since otherwise, we would have that $\vec{a} \in V(F_{\mathcal{I}},A)$ for some $\mathcal{I}  \subseteq \mathcal{I}_F$.  Next, since $\deg(F) = d$,  we see that for any $\vec{v} \in \mathcal{I}_F$ and for any $\vec{a} \in \mathcal{G}$,  the element $\vec{a}^{\vec{v}} \in A^{(i)}$ for some $1\leq i \leq d$.  From the hypothesis of Lemma \ref{spt}, we see that  the sets $A,A^{(2)}, \dots, A^{(d)}$ are contained in $\Gamma$.  Thus,  our strategy now is to bound, more generally, the number of solutions to equations of the form $m = \sum_{\vec{v} \in \mathcal{I}_G} c_{\vec{v}} z_{\vec{v}}$,  where each $z_{\vec{v}} \in A^{(i)}$ for some $1 \leq i \leq d$, as well as where there are no vanishing subsums.  Moreover,  for each such solution,  we will then bound the number of possibilities of $\vec{a} \in \mathcal{G}$ such that $\vec{a}^{\vec{v}} = z_{\vec{v}}$ for every $\vec{v} \in \mathcal{E}$. Thus,  writing 
\[ Y = A \cup A^{(2)} \cup \dots \cup A^{(d)}, \]
 we see that
\begin{equation} \label{jar}
 \sum_{\vec{a} \in \mathcal{G}} \mathds{1}_{m = \sum_{\vec{v} \in \mathcal{I}_{F}} c_{\vec{v}} \vec{a}^{\vec{v}} }   
 \leq    
 \sum_{\vec{v} \in \mathcal{I}_{F}} \sum_{z_{\vec{v}} \in Y}{}^{*}   \mathds{1}_{m = \sum_{\vec{v} \in \mathcal{I}_{F}} c_{\vec{v}} z_{\vec{v}}} \  
 \sigma( (z_{\vec{v}} )_{\vec{v}\in \mathcal{E}},  \mathcal{E},   A)   ,
 \end{equation}
where $ \sum_{\vec{v} \in \mathcal{I}_{F}} \sum_{z_{\vec{v}} \in Y}{}^{*}$ means that we are only considering $\{z_{\vec{v}}\}_{\vec{v} \in \mathcal{I}_F}$ which have no vanishing subsums of the form $\sum_{\vec{v} \in \mathcal{I} \subseteq \mathcal{I}_F} c_{\vec{v}} z_{\vec{v}} = 0$,  for any $\emptyset \neq \mathcal{I} \subseteq \mathcal{I}_F$.  Since $N(\vec{v}) \leq d$ for every $\vec{v} \in \mathcal{E}$,  we may combine Lemma $\ref{c2}$ along with $\eqref{jar}$ to deduce that
\[ \sum_{\vec{a} \in \mathcal{G}} \mathds{1}_{m = \sum_{\vec{v} \in \mathcal{I}_{F}} c_{\vec{v}} \vec{a}^{\vec{v}} }   \ll_{n,d}      \sum_{\vec{v} \in \mathcal{I}_{F}} \sum_{z_{\vec{v}} \in Y}{}^{*}   \mathds{1}_{m = \sum_{\vec{v} \in \mathcal{I}_{F}} c_{\vec{v}} z_{\vec{v}}} .\] 
Furthermore,  since $m \neq 0$ and since $Y$ lies in the multiplicative subgroup $\Gamma$ of $\mathbb{C}^{\times}$ with rank $r$,  we may now apply Lemma $\ref{subs}$, or more specifically, the discussion following Lemma $\ref{subs}$, to obtain the  estimate
\[ \sum_{\vec{a} \in \mathcal{G}} \mathds{1}_{m = \sum_{\vec{v} \in \mathcal{I}_{F}} c_{\vec{v}} \vec{a}^{\vec{v}} }   \ll_{d,n}  2^{O_{d,n}(r)} .  \]
This finishes the proof of Lemma \ref{spt}.
\end{proof}


\section{Proofs of Theorems \ref{main} and \ref{env}}

In this section, we will put forth the proofs of Theorems \ref{main} and \ref{env}. First, we will prove Theorem \ref{main} via an amalgamation of Lemmata \ref{freiman} and \ref{spt}.

\begin{proof}[Proof of Theorem \ref{main}]
We begin by noting that $F$ is non-degenerate, and so, Proposition \ref{rez} implies that $d(F) = n$. Moreover, we may assume that $F(0,\dots,0) = 0$, since that simply amounts to translating the set $F(A,\dots,A)$, which does not affect its cardinality.  Next,  if $|A| = 1$,  then the desired conclusion is trivial,  whence,  we may assume that $|A| \geq 2$.  This allows us to assume that $A \subseteq \mathbb{C}^{\times}$,  by passing,  if necessary,  to a subset $A' \subseteq A \cap \mathbb{C}^{\times}$ with $|A'| \geq |A|/2$.  Now, since $|A \cdot A| \leq K|A|$, we apply Lemma \ref{freiman} for the group $G = \mathbb{C}^{\times}$ to obtain some positive integer $l \ll (\log 2K)^{3 +o(1)}$ and a multiplicative $l$-dimensional centred convex coset progression $X$ with $|X| \leq e^l |A|$ such that $A$ is contained in a union of at most $e^l$ translates of $X$. Thus, there exists some $y \in \mathbb{C}^{\times}$ such that
\begin{equation} \label{smo}
|A \cap y \cdot X| \geq |A|/e^l.
\end{equation}
Let $A_1  = A \cap y\cdot X$. From the definition of $X$, we have that $X = P \cdot H$, where $H$ is a finite subgroup of $\mathbb{C}^{\times}$ and $P$ is a multiplicative $l$-dimensional centred convex progression in $\mathbb{C}^{\times}$. Thus, $P$ is contained in an multiplicative subgroup of $\mathbb{C}^{\times}$ of rank at most $l$. Moreover, since $H$ is a finite subgroup, we see that the set $y\cdot X = y\cdot P\cdot H$ is contained in a multiplicative subgroup of rank at most $l+1$. This implies that the set $A_1$ is contained in some multiplicative subgroup $\Gamma \subseteq \mathbb{C}^{\times}$ of rank at most $l+1 \ll (\log 2K)^{3 +o(1)}$. 

We now apply Lemma \ref{spt} to procure some $\mathcal{G}\subseteq A_1^n$ such that $|A_1^n \setminus \mathcal{G}| \ll_{d,n} |A_1|^{n-1}$ and
\[ \sup_{m \in \mathbb{C}} r_{F, \mathcal{G}}(m) \ll_{d,n} 2^{O_{d,n}((\log 2K)^{3 +o(1)}) } , \]
where the representation function $r_{F, \mathcal{G}}$ is defined as in \eqref{repdef}.  On the other hand, one may apply double counting to note that
\[ \sum_{m \in F(A_1,\dots,A_1)}  r_{F, \mathcal{G}}(m)  = |\mathcal{G} | \gg |A_1|^n, \]
which, in turn, combines with the preceding bound to give us
\[ |F(A_1, \dots, A_1)| \gg \frac{|A_1|^n}{  \sup_{m \in \mathbb{C}} r_{F, \mathcal{G}}(m)} \gg_{d,n}  \frac{|A_1|^n}{    2^{O_{d,n}((\log 2K)^{3 +o(1)}) }} . \]
Noting $\eqref{smo}$ along with the fact that $F(A,\dots, A) \supseteq F(A_1,\dots,A_1)$, we get that
\[  |F(A,\dots,A)| \geq |F(A_1, \dots, A_1)| \gg_{d,n}  \frac{|A|^n}{    2^{O_{d,n}((\log 2K)^{3 +o(1)}) }} ,  \]
which is the claimed result.
\end{proof}

Our next aim is to prove Theorem \ref{env}, and we begin this by first furnishing the following lemma about higher dimensional sumsets in $\mathbb{Z}/N\mathbb{Z} \times \mathbb{Z}^d$.

\begin{lemma} \label{sumset}
    Let $d, N$ be natural numbers, let $X$ be a finite subset of $G = \mathbb{Z}/N\mathbb{Z}\times \mathbb{Z}^d$ with $0 \in X$ such that $X$ generates $G$ and $|X+X| \leq K|X|$, for some $K \geq 1$. Then $d \leq 16K$.
\end{lemma}

\begin{proof}
Given $1 \leq j \leq 4$, we define $I_j \subseteq \mathbb{Z}/N\mathbb{Z}$ to be the set satisfying
\[ I_j = \{0,1,2,\dots, N -1 \} \cap [(j-1)N/4, jN/4) \ \  (\mod N) , \]
and we denote $X_j = X \cap (I_j \times \mathbb{Z}^d)$. Note that $|X| = |X_1| + \dots + |X_4|$, whence, there exists some $1 \leq j \leq 4$ such that $|X_j| \geq |X|/4$. We now construct a Freiman $2$-isomorphism $\vartheta:X_j \to \mathbb{Z}^{d+1}$ and we consider the set $Y = \vartheta(X_j)$. Denoting $r$ to be the dimension of the smallest affine subspace of $\mathbb{Z}^{d+1}$ containing $Y$, we see that $1 \leq r \leq d+1$ as well as $|Y| \geq r$.  We may now apply Lemma \ref{frlem} to deduce that
\begin{equation} \label{zdf}
    |Y+Y| \geq (r+1)|Y| - r(r+1)/2.  
\end{equation} 

We divide our proof into two cases, the first being when $r \geq d/2$. In this case, we note that
\[ |Y + Y| = |X_j + X_j| \leq |X+X| \leq K|X|, \]
which then combines with \eqref{zdf} to dispense the estimate
\begin{align*} 2K|X| & \geq 2(r+1)( |Y| - r/2) \geq 2(r+1) |Y|/2 \\
& \geq (d+2)|Y|/2 \geq (d+2)|X|/8. 
\end{align*}
This gives us that $d \leq 16K,$ which is the desired bound. Our second case is when $r < d/2$. In this case, note that there exist $d-r$ linearly independent elements $z_1, \dots, z_{d-r} \in X$ which lie outside the affine span of $X_j$.  Thus we have that
\[ K|X| \geq |X+X| \geq \sum_{i=1}^{d-r} |X_j + z_i| = (d-r)|X_j| > d|X|/8 , \]
which supplies the claimed bound $d \leq 16K$.
\end{proof}

With this lemma in hand, we now present our proof of Theorem \ref{env}.

\begin{proof}[Proof of Theorem \ref{env}]
Note that we may assume that $F(0,\dots,0) = 0$ simply by replacing the polynomial $F(\vec{x})$ with the polynomial $F(\vec{x}) - F(0,\dots, 0)$. Moreover, applying Proposition \ref{rez}, we see that since $F$ is non-degenerate, we must have $d(F) = n$.  Now, let $A\subseteq \mathbb{C}^{\times}$ be a finite set with $|A\cdot A|\leq K|A|$. We may assume that $1 \in A$ since we can replace $A$ by $A' = a^{-1}\cdot A$ for some $a \in A$ and note that
\[ F(A, \dots, A) = F'(A' ,\dots, A') ,\]
where $F' \in \mathbb{C}[x_1,\dots,x_n]$ satisfies $F'(x_1, \dots, x_n) = F(a x_1, \dots, a x_n)$. Here, it is important to note that $F'(0,\dots,0) = 0$ and $d(F') = n$ and $|A'\cdot A'| \leq K|A'|$. Now, let $\mathcal{V}$ be the multiplicative subgroup generated by $A$. Note that $\mathcal{V}$ is isomorphic to $\mathbb{Z}^r \times H$, for some $r \in \mathbb{N}\cup\{0\}$ and for some finite multiplicative subgroup $H$ of $\mathbb{C}^{\times}$. Since all the finite multiplicative subgroups of $\mathbb{C}^{\times}$ are isomorphic to cyclic groups, we see that $\mathcal{V}$ is isomorphic to $\mathbb{Z}^r \times \mathbb{Z}/N\mathbb{Z}$ for some $N \in \mathbb{N}$. We now apply Lemma \ref{sumset} to deduce that $r \leq 16K$, that is, $A$ is contained in a multiplicative subgroup of $\mathbb{C}^{\times}$ of rank at most $16K$. We now apply Lemma \ref{spt} to procure the claimed inequality \eqref{mo}. Furthermore, \eqref{mo} supplies the second inequality stated in Theorem \ref{env} in a straightforward manner since
\begin{align*}
    |\{ (\vec{a},\vec{a}') \in \mathcal{G}\times\mathcal{G} : F(\vec{a}) = F(\vec{a}') \}|  
    & = \sum_{m \in \mathbb{C}} r_{F,\mathcal{G}}(m)^2 \\ 
    & \ll_{d,n} 2^{O_{d,n}(K)} \sum_{m \in \mathbb{C}} r_{F,\mathcal{G}}(m)  \\
    & \leq 2^{O_{d,n}(K)} |A|^n. \qedhere
\end{align*}  
\end{proof}


\section{Proof of Theorem \ref{condn}}

We utilise this section to present the connections between our main results and the well-known polynomial Freiman-Ruzsa conjecture. As mentioned before, the latter concerns obtaining quantitatively optimal dependence on the doubling $K$ in Freiman's theorem and is a central open problem in additive combinatorics. In particular, it essentially asserts that one can have $l \ll \log 2K$ in the statement of Lemma \ref{freiman}, see for instance \cite[Conjecture 1.5]{Sa2013}. This is a very deep problem with a variety of applications to topics in additive combinatorics and analytic number theory, see \cite[\S13]{Sa2013}, and in particular, it is known that such conjectures imply very strong results towards the sum-product problem \cite{Ch2009}. We will now present one of the weaker versions of this conjecture as presented in \cite{Ch2009}, see also \cite{Ma2015, GMT2023}.

\begin{Conjecture} \label{wkpf}
Let $V$ be a $\mathbb{Z}$-module, let $A \subseteq V$ be a finite set with $|A+A| \leq K|A|$ for some $K \geq 1$. Then there exists $A_1 \subseteq A$ with $|A_1| > |A|/K^c$ and elements $\xi_1, \dots, \xi_d \in V$ for some $d < c \log 2K$ such that
\[ A_1 \subseteq \mathbb{Z} \xi_1 + \dots + \mathbb{Z} \xi_d, \]
where $c>0$ is some absolute constant. 
\end{Conjecture}

As in \cite{Ch2009}, we will now show that Conjecture \ref{wkpf} implies a much more stronger version of Theorem \ref{main} for sets of real numbers.

\begin{Proposition} \label{chs}
Suppose that Conjecture \ref{wkpf} holds true. Then for every $d,n \in \mathbb{N}$, for every $K \geq 1$ and  for every non-degenerate $F \in \mathbb{C}[x_1,\dots,x_n]$ with $\deg F = d$  and for every finite set $A \subseteq \mathbb{R}$ with $|A \cdot A| \leq K|A|$, one has
\[ |F(A, \dots, A)| \gg_{d,n} \frac{|A|^n}{(2K)^{O_{d,n}(1)}}.  \]
\end{Proposition}

\begin{proof}
As in the proof of Thereom \ref{main}, we may apply Proposition \ref{rez} to assume that $d(F) = n$ and further assume that $F(0, \dots, 0) = 0$. Now, since $A$ is a finite set of real numbers, we can write $A = A_1 \cup A_2 \cup A_3$, where $A_1 = A \cap (0, \infty)$ and $A_2 = A \cap (-\infty,0)$ and $A_3 = A \cap \{0\}$. We see that either $|A_1| \geq |A|/3$ or $|A_2| \geq |A|/3$. Suppose that $|A_1| \geq |A|/3$; the other case follows very similarly. Now, note that 
\[ 3K|A_1| \geq K|A| \geq |A\cdot A|   \geq |A_1 \cdot A_1|. \]
Moreover, since $A_1$ generates a multiplicative subgroup $V$ of $\mathbb{R}^{\times}$ which is isomorphic to $\mathbb{Z}^r$ for some $r \in \mathbb{N}$, we may now apply Conjecture \ref{wkpf} to deduce the existence of some $A_4 \subseteq A_1$ with $|A_4| \gg |A|/K^c$ such that the multiplicative subgroup $\Gamma$ generated by $A_4$ has rank at most $ c \log 2K$, for some absolute constant $c>0$. Applying Lemma \ref{spt} now, we may find some $\mathcal{G} \subseteq A_4^n$ with $|A_4^n \setminus \mathcal{G}| \ll_{d,n} |A_4|^{n-1}$ such that
\[ \sup_{m \in \mathbb{C}} r_{F, \mathcal{G}}(m) \ll_{d,n} 2^{O_{d,n}(\log 2K)} \ll (2K)^{O_{d,n}(1)}. \]
We conclude our proof by observing, as before, that
\begin{align*}
 |F(A,\dots, A)| 
 & \geq  |F(A_4,\dots, A_4)| \geq \frac{|\mathcal{G}|}{ \sup_{m \in \mathbb{C}} r_{F, \mathcal{G}}(m)}   \\
 & \gg_{d,n} \frac{|A_4|^n}{ (2K)^{O_{d,n}(1)}} \gg_{d,n} \frac{|A|^n}{ (2K)^{O_{d,n}(1)}} . \qedhere
 \end{align*}
\end{proof}

We end this section by presenting the proof of Theorem \ref{condn}.

\begin{proof}[Proof of Theorem \ref{condn}]
As in the proof of Proposition \ref{chs}, we may assume that $A \subseteq (0,\infty)$ by passing to a large subset and replacing $F(\vec{x})$ with $F(-\vec{x})$ if necessary. Moreover, we may assume that $F(0,\dots, 0) = 0$ and apply Proposition \ref{rez} to deduce that $d(F) = k$. Let $l \geq k$ be some positive integer that we will fix later and suppose that $|A^{(2^l)}| \leq |A|^k$. Then we have that
\[ \prod_{i=0}^{l-1} \frac{|A^{(2^{i+1})}|}{|A^{(2^{i})}|} = \frac{|A^{(2^l)}|}{|A|}  < |A|^k,\]
whereupon, there exists some $0 \leq i \leq l-1$ such that
\[ \frac{|A^{(2^{i+1})}|}{|A^{(2^{i})}|}  < |A|^{k/l} .  \]
Writing $B = A^{(2^i)}$, we see that $B \subseteq (0,\infty)$ is a finite set satisfying $|B\cdot B| \leq |A|^{k/l}|B|.$ As in the proof of Proposition \ref{chs}, we may now apply Conjecture \ref{wkpf} to deduce the existence of some $B' \subseteq B$ with $|B'| \gg |B|/|A|^{ck/l}$ such that the multiplicative subgroup $\Gamma$ generated by $B'$ has rank at most $ ck l^{-1} \log 2|A|$, for some absolute constant $c>0$. Moreover, note that 
\[  |B\cdot B'| \leq |B\cdot B| \leq |A|^{k/l}|B| \ll |A|^{(c+1)k/l}|B'|, \]
whence, we may apply Lemma \ref{rsco} to see that $B \subseteq T \cdot B' \cdot (B')^{-1}$ where $T \subseteq B$ is some non-empty set satisfying $|T| \ll |A|^{(c+1)k/l}$. On the other hand, note that $B$ contains $\lambda \cdot A$ for some $\lambda \neq 0$. Thus, for some $x \neq 0$, we have that
\[   |A \cap (x \cdot B' \cdot (B')^{-1})|  \geq |A|/|T| \gg |A|^{ 1- (c+1)k/l}. \]
We now observe that the set $x \cdot B' \cdot (B')^{-1}$ is contained in a multiplicative subgroup of rank at most $ck l^{-1} \log 2|A| + 1$, and so, upon writing $A' = A \cap (x \cdot B' \cdot (B')^{-1})$, we may apply a quantitatively precise version of Lemma \ref{spt} to deduce that
\[ |F(A',\dots,A')| \gg_{d,k} |A'|^{2k} 2^{-O_{d}(k^6 (ck l^{-1} \log 2|A| + 1)) } \gg_{d,k} |A|^{k - (2c+2)k^2/l  - O_d(ck^7/l )} . \]
Choosing $l = \ceil{\varepsilon^{-1}C_{d}k^8}$ for some large constant $C_d>0$, we may deduce that
\[ |F(A,\dots,A)| \geq |F(A',\dots,A')| \gg_{d,k} |A|^{k - \varepsilon}. \]
Putting everything together, we now see that either 
\[ |A^{(2^l)}| > |A|^k \ \text{or} \ |F(A,\dots,A)| \gg_{d,k} |A|^{k- \varepsilon}, \]
whence, setting $n=2^l$ delivers the desired conclusion.
\end{proof}


\bibliographystyle{amsbracket}
\providecommand{\bysame}{\leavevmode\hbox to3em{\hrulefill}\thinspace}

\end{document}